\documentclass[12pt]{article}
\usepackage{amssymb,oldlfont}
\usepackage{amsmath}
\usepackage{hyperref}
\usepackage{color}
\definecolor{myorange}{RGB}{180,90,0}
\definecolor{mygreen}{RGB}{70,140,0}
\usepackage[normalem]{ulem}

\usepackage[T1]{fontenc}
\usepackage[utf8]{inputenc}
\def\wrtext#1{\relax\ifmmode{\leavevmode\hbox{#1}}\else{#1}\fi}
\def\abs#1{\left|#1\right|}
\def\begeq{\begin{equation}}
\def\endeq{\end{equation}}

\def\part#1{\frac{\partial}{\partial #1}}

\def\norm#1{||\,#1\,||}

\newcommand{\real}{\mbox{\bf R}}
\newcommand{\comp}{\mbox{\bf C}}

\newtheorem{dref}{Definition}[section]

\newtheorem{theo}[dref]{Theorem}
\newtheorem{prop}[dref]{Proposition}

\newenvironment{proof}{\vspace{.3cm}\noindent{{\em Proof:}}}{\hfill$\Box$}

\title{A direct approach to the analytic Bergman projection}
\author{Alix \textsc{Deleporte} \footnote{Institute of Mathematics,
    University of Z\"urich, Winterthurerstrasse 190, CH-8057 Z\"urich, Switzerland, {\sf alix.deleporte@math.uzh.ch}} \and
  Michael \textsc{Hitrik} \footnote{Department of Mathematics, University of California, Los Angeles CA 90095-1555, USA, {\sf hitrik@math.ucla.edu}} \and Johannes \textsc{Sj\"ostrand}\footnote{IMB, Universit\'e de Bourgogne 9, Av. A. Savary, BP 47870
FR-21078 Dijon, France and UMR 5584 CNRS, {\sf johannes.sjostrand@u-bourgogne.fr}}}

\date{}

\begin{document}
\maketitle

\vspace*{1cm}
\noindent
{\bf Abstract}:
We develop a direct approach to the semiclassical asymptotics for Bergman projections in exponentially weighted spaces of holomorphic functions, with real analytic strictly plurisubharmonic weights. In particular, the approach does not rely upon the Kuranishi trick and it allows us to shorten and simplify proofs of a result due to~\cite{RSV} and~\cite{Del}, stating that in the analytic case, the amplitude of the asymptotic Bergman projection is a realization of a classical analytic symbol.

\section{Introduction}
\setcounter{equation}{0}
\label{sec_introduction}

Let $\Omega$ be a strictly pseudoconvex domain in $\comp^n$ and let $\Phi \in C^{\infty}(\Omega;\real)$ be a strictly plurisubharmonic function
(i.e. the Hermitian matrix $\partial\overline{\partial}\Phi$ is positive definite everywhere in $\Omega$). The study of the exponentially weighted $L^2$--space of holomorphic functions
$$
H_{\Phi}(\Omega) = \left\{u:\Omega\to \comp\text{ holomorphic}; \int_{\Omega}|u|^2e^{-\frac{2}{h}\Phi} < \infty\right\},
$$
with a small parameter $h>0$, plays a basic role in complex analysis. In particular, it serves as a local model for the space of holomorphic sections of a high power of an ample line bundle over a complex manifold. In this article, we are interested in the asymptotic description, in the semiclassical limit $h\to 0^+$, of the orthogonal projection $\Pi:L^2(\Omega; e^{-2\Phi/h})\to H_{\Phi}(\Omega)$ and its integral kernel. The Bergman projection $\Pi$ can be studied in many different ways, sharing as a common core the spectral gap property for the $\overline{\partial}$--operator on $L^2(\Omega;e^{-2\Phi/h})$, or rather for the corresponding Hodge Laplacian, as established in~\cite{Ho65}. The spectral gap implies directly that the Bergman kernel is rapidly decreasing away from the diagonal~\cite{Christ},~\cite{Delin}. The existence of a complete asymptotic expansion in powers of $h$ for the Bergman kernel has been shown in~\cite{Catlin},~\cite{Z98}, by means of a reduction to the main result of~\cite{BoSj} on the asymptotic behavior of the Szeg\H{o} kernel on the boundary of a strictly pseudoconvex smooth domain. The work~\cite{BBSj} has subsequently provided a self-contained proof of the existence of the expansion, by constructing local asymptotic Bergman kernels directly, using some of the ideas of analytic microlocal analysis, developed in~\cite{Sj82}; see also~\cite{MelSj}. Other self-contained strategies for the study of the Bergman kernel and its generalizations include~\cite{Ma_Mar},~\cite{KMM}.

\medskip
\noindent
The case of a real analytic weight $\Phi$ has been the subject of a recent intense activity~\cite{HLX,RSV,Del,Cha,HX}. In
this setting, one shows that the amplitude in the asymptotic Bergman kernel is a realization of a classical analytic symbol, in the sense
of~\cite{BoKr,Sj82}, and one can describe the Bergman projection $\Pi$ up to an exponentially small error, ${\cal O}(e^{-\frac{1}{Ch}})$, for some $C>0$. In~\cite{RSV}, an essential ingredient in the proof of this result consists in exploiting the Kuranishi trick. This ingredient is already present in~\cite{BBSj}, see the discussion following (2.7) there. The alternative strategy used in~\cite{Del,Cha,HX} is a direct verification that the coefficients in the complete expansion of the Bergman kernel amplitude form a classical analytic symbol. Both strategies become somewhat problematic when the Levi form $\partial\overline{\partial}\Phi$ of $\Phi$ becomes degenerate or nearly degenerate at a point or along a submanifold. A natural occurrence of such a behavior appears in the work in progress~\cite{HiSj}, in the context of second microlocalization. See also~\cite{Mar-Sav}. A direct approach to Bergman projections, in particular not relying upon the Kuranishi trick, is therefore desirable, and it is precisely our purpose here to develop such an approach, in the real analytic case

\medskip
\noindent
The following is the main result of this work. 
\begin{theo}\label{Theorem1}
Assume that $\Phi$ is real analytic in $\Omega$, and let $x_0\in \Omega$. There exist a unique classical analytic symbol $a(x,\widetilde{y};h)$, defined in a neighborhood of $(x_0,\overline{x_0})$, solving
\begeq
\label{eq1.1}
(A a)(x,\widetilde{y};h) = 1, 
\endeq
where $A$ is an elliptic analytic Fourier integral operator, and small open neighborhoods $U \Subset V \Subset \Omega$ of $x_0$, with $C^{\infty}$--boundaries, such that the operator
\begeq
\label{eq1.2} 
\widetilde{\Pi}_{V} u(x) = \frac{1}{h^n} \int_{V} e^{\frac{2}{h}\Psi(x,\overline{y})} a(x,\overline{y};h)  u(y) e^{-\frac{2}{h}\Phi(y)}\, L(dy)
\endeq
satisfies
\begeq
\label{eq1.3}
\widetilde{\Pi}_V - 1 = {\cal O}(1) e^{-\frac{1}{Ch}}: H_{\Phi}(V) \rightarrow H_{\Phi}(U), \quad C >0.
\endeq
Here in {\rm (\ref{eq1.2})}, the holomorphic function $\Psi$ is the polarization of $\Phi$ and $L(dy)$ is the Lebesgue measure on $\comp^n$.
\end{theo}

\medskip
\noindent
Let us point out that the general strategy of constructing the amplitude of the asymptotic Bergman projection by inverting an elliptic analytic Fourier integral operator, acting on the space of analytic symbols, was also followed in~\cite{RSV}. That work proceeded by means of the Kuranishi trick, and the Fourier integral operator in question was obtained by composing various integral transforms. In contrast to~\cite{RSV}, in Theorem \ref{Theorem1}, we bypass the use of the Kuranishi trick and construct the operator $A$ in (\ref{eq1.1}) directly. This article can therefore be regarded as an alternative to the two approaches to asymptotic Bergman kernels mentioned above, and we plan to generalize it to degenerate situations as well. It seems also that the method for determining the amplitude in the Bergman kernel, consisting of solving the equation (\ref{eq1.1}), is quite direct.

\medskip
\noindent
The plan of the paper is as follows. In Section \ref{sec_resolution-identity} we review a resolution of the identity in the $H_{\Phi}$--spaces
related to the Fourier inversion formula in the complex domain. Section \ref{sec_four-integr-oper} is devoted to the construction of a suitable analytic symbol to be used as the amplitude for the asymptotic Bergman projection. We introduce a complex phase function, with no fiber variables present, such that the corresponding canonical transformation maps the zero section to itself. Associated to it is a Fourier integral operator $A$, and we define the amplitude $a$ in (\ref{eq1.1}) as the unique classical analytic symbol of order $0$ such that $Aa=1$, locally. Then, in Section \ref{sec_reproducing-property} we show that the operator $\widetilde{\Pi}_V$ given in (\ref{eq1.2}) satisfies the reproducing property in $H_{\Phi}$, locally and in the weak formulation: for $u,v\in H_{\Phi}(\Omega)$, on a small enough set $V$ we have
$(\widetilde{\Pi}_Vu,v)_{H_{\Phi}(V)}=(u,v)_{H_{\Phi}(V)}+{\cal O}(e^{-\frac{1}{Ch}})$, provided that $v$ is small near the boundary of
$V$. The proof consists of a contour deformation argument which depends on the resolution of the identity of Section \ref{sec_resolution-identity}. The contour deformation is first justified for elements of $H_{\Phi}$ sufficiently localised near a
point, and the decomposition of Section \ref{sec_resolution-identity} ensures that, by linearity, the reproducing property is true on the
whole of $H_{\Phi}$. In Section \ref{sec:end-proof}, we conclude the proof of Theorem \ref{Theorem1} using the $\overline{\partial}$-method.

\medskip
\noindent
Once the local approximate reproducing property of Theorem \ref{Theorem1} has been established, a global version (uniformly in any compact subset of $\Omega$, or uniformly on a complex compact manifold without boundary) follows from cut-and-paste arguments and, in particular, the $L^2$--estimates for the $\overline{\partial}$--operator. Such arguments have already been developed carefully in \cite{RSV}, and in this work we shall refrain therefore from repeating that discussion. We would finally like to emphasize that the majority of the methods and the ideas in this paper stem from~\cite{Sj82}.

\medskip
\noindent
{\bf Acknowledgments}. Part of this material is based upon work supported by the National Science Foundation under Grant
No. DMS--1440140, while A.D. and M.H. were in residence at the Mathematical Sciences Research Institute in Berkeley, CA, during the Fall 2019 semester. The support of the work of A.D. by the Swiss Forschungskredit is also gratefully acknowledged.

\section{A resolution of the identity}
\label{sec_resolution-identity}
\setcounter{equation}{0}
Let $\Omega \subset \comp^n$ be open, and let $\Phi\in C^{\infty}(\Omega;\real)$ be strictly plurisubharmonic in $\Omega$: there exists $0 < c\in C(\Omega)$ such that
\begeq
\label{eq2.1}
\sum_{j,k=1}^n \frac{\partial^2 \Phi}{\partial x_j \partial \overline{x}_k}(x) \xi_j \overline{\xi}_k \geq c(x) \abs{\xi}^2, \quad x\in \Omega, \quad \xi \in \comp^n.
\endeq
Let us define the space 
\begeq
\label{eq2.4}
H_{\Phi}(\Omega) = {\rm Hol}(\Omega) \cap L^2(\Omega; e^{-2\Phi/h}L(dx)),
\endeq
equipped with its natural Hilbert space norm
\begeq
\label{eq2.4.1}
\norm{u}_{L^2_{\Phi}(\Omega)} = \left(\int_{\Omega} \abs{u(x)}^2 e^{-2\Phi(x)/h}\, L(dx)\right)^{1/2}. 
\endeq

\bigskip
\noindent
Let $x_0\in \Omega$ and let $V \Subset \Omega$ be an open neighborhood of $x_0$ with $C^{\infty}$--boundary. The strict plurisubharmonicity
of $\Phi$ has the following consequence.

\begin{prop}
\label{prop:good-contour}
There exists a small neighborhood $V\Subset \Omega$ of $x_0$, with $C^{\infty}$--boundary, such that the $2n$-dimen\-sional manifold $\Lambda(x) \subset \comp^{2n}_{y,\theta}$ given by
\begeq
\label{eq2.5}
\theta = \theta(x,y) = \frac{2}{i} \left(\frac{\partial \Phi}{\partial y}(y) + \frac{1}{2}\Phi''_{yy}(y)(x-y)\right), \quad y\in V,
\endeq
is a good contour for the plurisubharmonic function $(y,\theta) \mapsto -{\rm Im}\, ((x-y)\cdot \theta) + \Phi(y)$, for $x\in V$, in the sense of
{\rm \cite[Chapter 3]{Sj82}}: it is maximally totally real and such that there exists $\delta >0$ such that for all $x,y\in V$, we have
\begeq
\label{eq2.6}
-{\rm Im}\, ((x-y)\cdot \theta) + \Phi(y) \leq \Phi(x) - \delta \abs{x-y}^2.
\endeq
Moreover, the contour $\Lambda(x)$ depends holomorphically on $x \in V$.
\end{prop}
\begin{proof}
The estimate \eqref{eq2.6} is a direct consequence of \eqref{eq2.1} and Taylor's formula. To see that the $2n$-dimensional $C^{\infty}$--submanifold $\Lambda(x)$ (with $C^{\infty}$--boundary) is maximally totally real, we use the following general observation: let $q$ be a plurisubharmonic quadratic form on $\comp^n$, and let $L \subset \comp^n$ be a real linear subspace of dimension $n$ such that $q|_L$ is negative definite. Then $L$ is maximally totally real, see~\cite[Proposition 3.1]{Sj82}.
\end{proof}

\bigskip
\noindent
Let $V_1 \Subset V_2 \Subset V$ be open neighborhoods of $x_0$ and let $\chi \in C^{\infty}_0(V; [0,1])$ be such that $\chi = 1$ near $\overline{V_2}$. Following~\cite[Chapter 3]{Sj82},~\cite{BBSj}, we have the following result, representing the identity operator on $H_{\Phi}(V)$ as a pseudodifferential operator in the anti-classical quantization.

\begin{prop}
\label{Fourier_inv}
Let $V_1$ and $\Lambda(x)$ be as above. There exists $\eta > 0$ such that when $u\in H_{\Phi}(V)$, we have for $x\in V_1$,
\begeq
\label{eq2.7}
u(x) = \frac{1}{(2 \pi h)^n} \int\!\!\!\int_{\Lambda(x)} e^{\frac{i}{h}(x-y)\cdot \theta} u(y)\chi(y)\, dy\,d\theta + {\cal O}(1)\norm{u}_{L^2_{\Phi}(V)} e^{\frac{1}{h}(\Phi(x) - \eta)}.
\endeq
Here it is assumed that the contour $\Lambda(x)$ has been equipped with a suitable orientation.
\end{prop}
\begin{proof}
Following~\cite{BBSj}, the proof proceeds by applying the Stokes formula to the $(2n,0)$--form
$$
\frac{1}{(2 \pi h)^n} e^{\frac{i}{h}(x-y)\cdot \theta} u(y)\chi(y)\, dy\,\wedge d\theta,
$$
integrated over the (oriented) boundary of the $(2n+1)$--dimensional chain given by
$$
V \times [0,s] \ni (y,\lambda) \mapsto (y,\theta(x,y) + i\lambda(\overline{x-y}))\in \comp^{2n}_{y,\theta},
$$
and letting $s\rightarrow \infty$.
\end{proof}

\medskip
\noindent
{\it Remark.} In particular, the resolution of identity given by (\ref{eq2.7}) is valid for $u\in H_{\Phi}(\Omega)$.

\bigskip
\noindent
It follows from Proposition \ref{Fourier_inv} that, for some $\eta>0$ and for all $u\in H_{\Phi}(V)$, we have 
\begeq
\label{eq2.8}
u(x) = \int_V u_y(x)\, dy\, d\overline{y} + {\cal O}(1) \norm{u}_{L^2_{\Phi}(V)} e^{\frac{1}{h}(\Phi(x) - \eta)}, \quad x\in V_1,
\endeq
with
\begeq
\label{eq2.9}
u_y(x) = \frac{1}{(2 \pi h)^n} e^{\frac{i}{h}(x-y)\cdot \theta(x,y)} u(y)\chi(y) {\rm det}\,(\partial_{\overline{y}}\theta(x,y)) \in 
H_{F_y}(V),
\endeq
where $F_y$ is strictly plurisubharmonic such that 
\begeq
\label{eq2.10}
F_y(x) \leq \Phi(x) - \delta \abs{x-y}^2, \quad \delta > 0.
\endeq

\medskip
\noindent
We conclude this section with a pointwise estimate for elements of $H_{\Phi}(V)$. 

\begin{prop}
\label{prop:ctrl_ptwise}
Let $V_1\Subset V\Subset \Omega$. Then there exists $C>0$ such that for all $u\in H_{\Phi}(V)$ and for all $h\in (0,1]$, we have 
\begeq
\label{eq2.11}
\sup_{V_1}|ue^{-\Phi/h}|\leq Ch^{-n}\|u\|_{H_{\Phi}(V)}. 
\endeq
\end{prop}
\begin{proof}
A holomorphic function is equal to its mean value over an open ball, so that, for all $x\in V_1$ and all $h>0$ small enough so that
$B(x,h)\subset V$, we have  
$$
u(x) = \frac{C_n}{h^{2n}}\int_{|y-x|<h} u(y)\,L(dy). 
$$
Here $C_n >0$ depends on $n$ only. It follows that 
\begin{multline*}
\abs{u(x)} e^{-\Phi(x)/h} \leq \frac{C_n}{h^{2n}}\int_{|y-x|<h} \abs{u(y)} e^{-\Phi(x)/h}\, L(dy) \\
\leq \sup_{|y-x|<h}e^{(\Phi(y)-\Phi(x))/h} \frac{C_n}{h^{2n}}\int_{|y-x|<h}\abs{u(y)} e^{-\Phi(y)/h}\, L(dy) \\
\leq \frac{C'}{h^{2n}}\|u\|_{H_{\Phi}(V)}\|1\|_{L^2(B(x,h))} \leq \frac{C'}{h^n}\|u\|_{H_{\Phi}(V)}.
\end{multline*}
\end{proof}

\section{A Fourier integral operator with complex phase}
\label{sec_four-integr-oper}
\setcounter{equation}{0}
Assume that the strictly plurisubharmonic function $\Phi$ is real analytic in $\Omega$, and let $x_0 \in \Omega$. Associated to $\Phi$ is the polarization $\Psi(x,y)$, which is the unique holomorphic function of $(x,y)\in {\rm neigh}((x_0,\overline{x_0}),\comp^{2n})$ such that
\begeq
\label{eq3.1}
\Psi(x,\overline{x}) = \Phi(x),\quad x\in {\rm neigh}(x_0,\comp^n).
\endeq
The matrix $\Psi''_{xy}(x_0,\overline{x_0}) = \Phi''_{x \overline{x}}(x_0)$ is non-singular and the following classical estimate,
\begeq
\label{eq3.2}
\Phi(x) + \Phi(y) - 2{\rm Re}\, \Psi(x,\overline{y}) \asymp \abs{x-y}^2, \quad x,y\in {\rm neigh}(x_0,\comp^n),
\endeq
is implied by the strict plurisubharmonicity of $\Phi$, see for instance~\cite{RSV}.

\medskip
\noindent
Let us set
\begeq
\label{eq3.3}
\varphi(y,\widetilde{x}; x,\widetilde{y}) = \Psi(x,\widetilde{y}) - \Psi(x,\widetilde{x}) - \Psi(y,\widetilde{y}) + \Psi(y,\widetilde{x}).
\endeq
We have $\varphi\in {\rm Hol}({\rm neigh}((x_0,\overline{x}_0;x_0,\overline{x}_0),\comp^{4n}))$. Furthermore, at the point $(x_0,\overline{x_0}; x_0, \overline{x_0})$, the $2n\times 2n$--matrix of second derivatives
\begeq
\label{eq3.4}
\varphi''_{(y,\widetilde{x}), (x,\widetilde{y})} = \left( \begin{array}{ccc}
\varphi''_{yx} & \varphi''_{y\widetilde{y}} \\\
\varphi''_{\widetilde{x}x} & \varphi''_{\widetilde{x}\widetilde{y}}
\end{array} \right) = \left( \begin{array}{ccc}
0 & -\Psi''_{y\widetilde{y}}(y,\widetilde{y}) \\\
-\Psi''_{\widetilde{x} x}(x,\widetilde{x}) & 0
\end{array} \right)
\endeq
is invertible; thus this matrix is non-degenerate in a neighbourhood of $(x_0,\overline{x_0}; x_0, \overline{x_0})$. Therefore,
$\varphi(y,\widetilde{x};x,\widetilde{y})$ is a generating function for the canonical transformation
\begeq
\label{eq3.5}
\kappa: \left(x,\widetilde{y}; -\frac{2}{i} \partial_x \varphi, -\frac{2}{i} \partial_{\widetilde{y}} \varphi\right) \mapsto \left(y,\widetilde{x}; \frac{2}{i} \partial_y \varphi, \frac{2}{i} \partial_{\widetilde{x}} \varphi\right).
\endeq
\begin{prop}
\label{can_transf}
The canonical transformation $\kappa$ maps the zero section to the zero section, and we have
\begeq
\label{eq3.6}
{\rm det}\, \varphi''_{(x,\widetilde{y}),(x,\widetilde{y})}(x_0, \overline{x}_0; x_0, \overline{x}_0) \neq 0.
\endeq
\end{prop}

\begin{proof}
Using the invertibility of $\Psi''_{x\widetilde{x}}(x_0,\overline{x}_0)$ and (\ref{eq3.3}), we see that $\partial_x \varphi =
0 \Longleftrightarrow \widetilde{y} = \widetilde{x}$, as well as $\partial_{\widetilde{y}} \varphi = 0 \Longleftrightarrow x = y$, and therefore the unique critical point of $\varphi$ with respect to the variables $(x,\widetilde{y})$ is given by $x = y$, $\widetilde{y} = \widetilde{x}$. The corresponding critical value is equal to $0$. When proving the proposition, we may therefore simplify the notation by considering a holomorphic function $\varphi(z,w)$ defined near $(0,0)$ in $\comp^{2m}$, such that
\begeq
\label{eq3.7}
{\rm det}\,\varphi''_{zw}(0,0)\neq 0, \quad \varphi'_w(z,w) = 0 \Longleftrightarrow w = z, \quad \varphi(z,z) = 0.
\endeq
It follows that
\begeq
\label{eq3.8}
\varphi'_z(z,z) = \partial_z \left(\varphi(z,z)\right) = 0,
\endeq
and therefore the canonical transformation
\begeq
\label{eq3.9}
\kappa: (w, -\partial_w \varphi(z,w)) \mapsto  (z, \partial_z \varphi(z,w))
\endeq
maps the zero section $\{\eta = 0\}$ to the zero section $\{\xi = 0\}$. It only remains to check that ${\rm det}\, \varphi''_{ww}(0,0) \neq 0$, and to this end we observe that the differential of $\kappa$ at $(0,0)$ is given by
\begeq
\label{eq3.10}
(\delta_w, - \varphi''_{wz}\delta_z - \varphi''_{ww}\delta_w) \mapsto (\delta_z, \varphi''_{zz}\delta_z + \varphi''_{zw}\delta_w), \quad \varphi'' = \varphi''(0,0),
\endeq
where $\delta_z$ and $\delta_w$ are infinitesimal increments.
If $\varphi''_{ww}\delta_w = 0$, we get $d\kappa(0,0): (\delta_w,0) \mapsto (0, \varphi''_{zw} \delta_w)$, and it follows that $\delta_w = 0$.
\end{proof}

\bigskip
\noindent
We now introduce an elliptic analytic Fourier integral operator $A$ in the complex domain, defined in a neighbourhood of
$(x_0,\overline{x_0})$. This Fourier integral operator is associated to the canonical transformation $\kappa$ in (\ref{eq3.5}) and acts on the
space of analytic symbols $H_0^{{\rm loc}}$, defined in a neighborhood of $(x_0,\overline{x_0})$. Here we recall that the space of analytic symbols $H_0^{{\rm loc}}$ has been introduced in~\cite[Chapter 1]{Sj82}. To this end, for $(y,\tilde{x})$ in a neighbourhood of
$(x_0,\overline{x_0})$, we let $\Gamma(y,\widetilde{x}) \subset \comp^{2n}_{x,\widetilde{y}}$ be a good contour for the pluriharmonic
phase function $(x,\widetilde{y}) \mapsto {\rm Re}\, \varphi(y,\widetilde{x}; x,\widetilde{y})$, so that
$\Gamma(y,\widetilde{x})$ is a $2n$-dimensional contour passing through the critical point $(y,\widetilde{x})$ and depending holomorphically on $(y,\widetilde{x})$, such that along $\Gamma(y,\widetilde{x})$ we have
\begeq
\label{eq3.11}
{\rm Re}\, \varphi(y,\widetilde{x}; x,\widetilde{y}) \leq -\frac{1}{C}\abs{x-y}^2 - \frac{1}{C}\abs{\widetilde{y} - \widetilde{x}}^2.
\endeq
Given an analytic symbol $u(x,\tilde{y};h)$ defined near $(x_0,\overline{x_0})$, we set
\begeq
\label{eq3.12}
(Au)(y,\widetilde{x};h) = \frac{1}{h^n} \int\!\!\!\int_{\Gamma(y,\widetilde{x})} e^{\frac{2}{h} \varphi(y,\widetilde{x}; x,\widetilde{y})} u(x,\widetilde{y};h) dx d\widetilde{y},
\endeq
so that $Au$ is an analytic symbol defined in a neighborhood of $(x_0,\overline{x_0})$.

\medskip
\noindent
Before stating the main result of this section, following~\cite[Chapter 1]{Sj82}, let us recall the notion of a classical analytic symbol. Let $V \subset \comp^n$ be open, $a_k\in {\rm Hol}(V)$, $k=0,1,\ldots\,$, and assume that for every $\widetilde{V} \Subset V$, there exists $C = C_{\widetilde{V}} > 0$ such that
\begeq
\label{eq3.13}
\abs{a_k(x)} \leq C^{k+1} k^k,\quad x\in \widetilde{V}.
\endeq
The series $a(x;h) = \sum_{k=0}^{\infty} a_k(x) h^k$ is called a formal classical analytic symbol of order zero. We have a realization of $a$ on $\widetilde{V}$ given by
\begeq
\label{eq3.14}
a_{\widetilde{V}}(x;h) = \sum_{0 \leq k \leq (C_{\widetilde{V}} e h)^{-1}} a_k(x) h^k,
\endeq
so that $a_{\widetilde{V}}\in {\rm Hol}(\widetilde{V})$, $\abs{a_{\widetilde{V}}(x;h)} \leq C_{\widetilde{V}} e/(e-1)$.

\begin{prop}\label{prop:constr-a}
There is a unique classical analytic symbol of order zero $a(x,\tilde{y};h)$, defined in a neighbourhood of $(x_0,\overline{x_0})$
such that
\begeq
\label{eq3.14.1}
(Aa)(y,\widetilde{x};h) = 1+{\cal O}(e^{-\frac{1}{Ch}}),
\endeq
near $(x_0,\overline{x_0})$.
\end{prop}
\begin{proof}
In view of Proposition \ref{can_transf}, the Fourier integral operator $A$ in (\ref{eq3.12}) maps classical analytic symbols defined near
$(x_0, \overline{x}_0)$ to classical analytic symbols defined in a neighborhood of the same point, see~\cite[Chapter
4]{Sj82}. Furthermore, in view of the ellipticity of $A$, from~\cite[Theorem 4.5]{Sj82}, we know that there exists a microlocal inverse $B$ of $A$ having the form
\begeq
\label{eq3.15}
(Bb)(x,\widetilde{y};h) = \frac{1}{h^n} \int\!\!\!\int_{\Gamma_1(x,\widetilde{y})} e^{-\frac{2}{h} \varphi(y,\widetilde{x}; x,\widetilde{y})} d(y,\widetilde{x},x,\widetilde{y};h) b(y,\widetilde{x};h) dy\, d\widetilde{x}.
\endeq
Here $d(y,\widetilde{x},x,\widetilde{y};h)$ is an elliptic classical analytic symbol defined in a neighborhood of the point $(x_0, \overline{x}_0; x_0, \overline{x}_0)\in \comp^{4n}$, $b(y,\widetilde{x};h)$ is a classical analytic symbol defined near $(x_0, \overline{x}_0)$, and $\Gamma_1(x,\widetilde{y})$ is a good contour for the pluriharmonic function $(y,\widetilde{x}) \mapsto -{\rm Re}\, \varphi(y,\widetilde{x}; x, \widetilde{y})$. Setting
\begeq
\label{eq3.16}
a(x,\widetilde{y};h) = (B 1)(x,\widetilde{y};h),
\endeq
we obtain the desired classical analytic symbol defined in a neighborhood of $(x_0, \overline{x}_0)$.
\end{proof}

\bigskip
\noindent
{\it Remark}. The equation $(Aa)(y,\widetilde{x};h) = 1$ can be formally written as follows:
\begeq
\label{eq3.17}
\frac{1}{h^n} \int\!\!\!\int e^{\frac{2}{h}\left(\Psi(x,\widetilde{y}) - \Psi(x,\widetilde{x}) - \Psi(y,\widetilde{y})\right)} a(x,\widetilde{y};h)\, dx d\widetilde{y} = e^{-\frac{2}{h} \Psi(y,\widetilde{x})}.
\endeq
Introducing the formal elliptic Fourier integral operators
\begeq
\label{eq3.18}
({\cal A} u)(x) = \frac{1}{h^{n/2}} \int e^{\frac{2}{h} \Psi(x,\widetilde{y})} a(x,\widetilde{y};h) u(\widetilde{y})\, d\widetilde{y},
\endeq
\begeq
\label{eq3.19}
({\cal C} u)(\widetilde{x}) = \frac{1}{h^{n/2}} \int e^{-\frac{2}{h} \Psi(y,\widetilde{x})} u(y)\, dy,
\endeq
we can rewrite (\ref{eq3.17}) in the form,
\begeq
\label{eq3.20}
\int\!\!\!\int {\cal K}_{\cal C}(\widetilde{x},x) {\cal K}_{\cal A}(x, \widetilde{y}) {\cal K}_{\cal C}(\widetilde{y},y)\, dx d\widetilde{y} = {\cal K}_{\cal C}(\widetilde{x},y).
\endeq
Here ${\cal K}_{\cal A}$, ${\cal K}_{\cal C}$ are the integral kernels of ${\cal A}$, ${\cal C}$, respectively. The equation (\ref{eq3.14.1}) is therefore formally equivalent to the operator equation
\begeq
\label{eq3.21}
{\cal C}\circ {\cal A} \circ {\cal C} = {\cal C} \Longleftrightarrow {\cal A} \circ {\cal C} = 1.
\endeq
In Sections \ref{sec_reproducing-property}, \ref{sec:end-proof} below, we shall see that the operator of the form
$$
(\widetilde{\Pi} u)(x) = \frac{1}{h^n} \int\!\!\!\int e^{\frac{2}{h}\left(\Psi(x,\widetilde{y}) - \Psi(y,\widetilde{y})\right)} a(x,\widetilde{y};h) u(y)\, dyd\widetilde{y}
$$
enjoys the (approximate) reproducing property on $H_{\Phi}$, and the equation (\ref{eq3.21}) can therefore be regarded as a formal factorization of the asymptotic Bergman projection.

\section{The reproducing property in the weak formulation}
\label{sec_reproducing-property}
\setcounter{equation}{0}
\setcounter{equation}{0}
Let us recall from Section \ref{sec_resolution-identity} that $V\Subset \Omega$ is a small open neighborhood of a point $x_0\in
\Omega$, and shrinking $V$ if necessary, we may assume that the polarization $\Psi$ of the real analytic weight function $\Phi$, introduced in
(\ref{eq3.1}), as well as the classical analytic symbol $a$, given in Proposition \ref{prop:constr-a}, are defined in a neighborhood of the
closure of the open set $V\times \rho(V)$. Here $\rho(x) = \overline{x}$ is the complex conjugation map. 

\medskip
\noindent
We introduce the following operator of Bergman type,
\begeq
\label{eq4.1}
\widetilde{\Pi}_{V} u(x) = \frac{1}{h^n} \int\!\!\!\int_{\Gamma_V} e^{\frac{2}{h}(\Psi(x,\widetilde{y}) - \Psi(y,\widetilde{y}))} a(x,\widetilde{y};h) u(y)\, dy\, d\widetilde{y}, \quad u\in H_{\Phi}(V),
\endeq
where the contour of integration $\Gamma_V \subset V \times \rho(V)$ is given by
\begeq
\label{eq4.2}
\Gamma_V = \{\tilde{y}=\overline{y},\,\,y\in V\}.
\endeq
Here in (\ref{eq4.1}) we have also chosen a realization of $a$ on $V\times \rho(V)$. It follows from (\ref{eq3.2}), combined with the Schur test, that
\begeq
\label{eq4.3}
\widetilde{\Pi}_{V} = {\cal O}(1): H_{\Phi}(V) \rightarrow H_{\Phi}(V).
\endeq

\medskip
\noindent
The purpose of this section is to show that the operator $\widetilde{\Pi}_V$ satisfies a reproducing property, in the weak formulation. Specifically, we shall prove that for a convenient class of $(u,v)\in H_{\Phi}(V)$, the continuous sesquilinear form
\begeq
\label{eq4.3.1}
H_{\Phi}(V) \times H_{\Phi}(V) \ni (u,v) \mapsto (\widetilde{\Pi}_{V} u,v)_{H_{\Phi}(V)}
\endeq
agrees, modulo an exponentially small error, with the scalar product $(u,v)_{H_{\Phi}(V)}$. This result cannot be expected to hold if $u,v$ are general elements of $H_{\Phi}(V)$, since they might both concentrate near the boundary of $V$ where we have cut off the integral operator $\widetilde{\Pi}_V$.

\medskip
\noindent
The following is the main result of this section. It will be instrumental in Section \ref{sec:end-proof}, when proving Theorem \ref{Theorem1}.

\begin{theo}
\label{reproducing_scalar_product}
There exists a small open neighborhood $W \Subset V$ of $x_0$ with $C^{\infty}$--boundary such that for each $\Phi_1 \in C(\Omega; \real)$,
$\Phi_1 \leq \Phi$, with $\Phi_1 < \Phi$ on $\Omega \backslash \overline{W}$, there exists $C > 0$ such that for all $u\in H_{\Phi}(V)$, $v\in H_{\Phi_1}(V)$, we have
\begeq
\label{eq4.3.2}
(\widetilde{\Pi}_V u,v)_{H_{\Phi}(V)} = (u,v)_{H_{\Phi}(V)} + {\cal O}(1) e^{-\frac{1}{Ch}} \norm{u}_{H_{\Phi}(V)} \norm{v}_{H_{\Phi_1}(V)}.
\endeq
\end{theo}

\medskip
\noindent
When proving Theorem \ref{reproducing_scalar_product}, using also the notation of Section \ref{sec_resolution-identity}, we let $W \Subset V_1 \Subset V$ be an open neighborhood of $x_0$ with $C^{\infty}$--boundary, to be chosen small enough, and let $\Phi_1 \in C(\Omega; \real)$ be such that
\begeq
\label{eq4.4}
\Phi_1 \leq \Phi \,\, \wrtext{in}\,\, \Omega, \quad \Phi_1 < \Phi \,\, \wrtext{on}\,\, \Omega\backslash \overline{W}.
\endeq
We shall study the scalar product
\begeq
\label{eq4.5}
(\widetilde{\Pi}_{V} u,v)_{H_{\Phi}(V)} = \int_{V} \widetilde{\Pi}_{V} u(x) \overline{v(x)} e^{-2\Phi(x)/h}\, L(dx), \quad u\in H_{\Phi}(V),\,\, v\in H_{\Phi_1}(V),
\endeq
and let us first write, using (\ref{eq4.3}), (\ref{eq4.4}), and the Cauchy-Schwarz inequality,
\begeq
\label{eq4.6}
(\widetilde{\Pi}_{V} u,v)_{H_{\Phi}(V)} = \int_{V_1} \widetilde{\Pi}_{V} u(x) \overline{v(x)} e^{-2\Phi(x)/h}\, L(dx) +
{\cal O}(1) e^{-\frac{1}{C h}} \norm{u}_{H_{\Phi}(V)} \norm{v}_{H_{\Phi_1}(V)}.
\endeq
Here and in what follows we let $C>0$ stand for constants which may depend on $\Phi$, $\Phi_1$, but not on $u$, $v$. Let next $V_2$ be an open set such that $V_1\Subset V_2\Subset V$ and observe that in view of \eqref{eq3.2}, we have 
\begin{equation}
\label{eq4.6.1}
\norm{\widetilde{\Pi}_V(1-\chi_{V_2})u}_{L^2_{\Phi}(V_1)} \leq {\cal O}(1) e^{-\frac{1}{Ch}}\norm{u}_{H_{\Phi}(V)}.
\end{equation}
Here $\chi_{V_2}$ denotes the characteristic function of $V_2$. Using (\ref{eq4.6} and (\ref{eq4.6.1}), we may therefore write
\begeq
\label{eq4.6.2}
(\widetilde{\Pi}_{V} u,v)_{H_{\Phi}(V)} = \int_{V_1} \widetilde{\Pi}_{V_2} u(x) \overline{v(x)} e^{-2\Phi(x)/h}\, L(dx) +
{\cal O}(1) e^{-\frac{1}{C h}} \norm{u}_{H_{\Phi}(V)} \norm{v}_{H_{\Phi_1}(V)},
\endeq
where, similarly to (\ref{eq4.1}), we set
\begeq
\label{eq4.6.3}
\widetilde{\Pi}_{V_2} u(x) = \frac{1}{h^n} \int\!\!\!\int_{\Gamma_{V_2}} e^{\frac{2}{h}(\Psi(x,\widetilde{y}) - \Psi(y,\widetilde{y}))} a(x,\widetilde{y};h) u(y)\, dy\, d\widetilde{y}.
\endeq
The advantage of representing the scalar product $(\widetilde{\Pi}_{V} u,v)_{H_{\Phi}(V)}$ in the form (\ref{eq4.6.2}) is due to the fact that in the right hand side of (\ref{eq4.6.2}), both the integrations in $x$ and $y$ are confined to suitable relatively compact subsets of the open set $V$, where good pointwise estimates on the holomorphic functions $u$ and $v$ are available, in view of Proposition \ref{prop:ctrl_ptwise}.

\medskip
\noindent
We would next like to apply the resolution of the identity (\ref{eq2.8}) to the holomorphic function $v\in H_{\Phi_1}(V)$ in the integral in the right hand side of (\ref{eq4.6.2}). To this end, let us first observe that thanks to the exponential decay of $v$ in $H_{\Phi}(V)$ away from $\overline{W}$, committing an exponentially small error, we may restrict the domain of integration in the right hand side of (\ref{eq2.8}) to an arbitrarily small but fixed neighborhood $W_1$ of $\overline{W}$, $W_1 \Subset V_1$. In precise terms, we may write
\begeq
\label{eq4.7}
v(x) = \int_{W_1} v_z(x)\, dz\,d\overline{z} + {\cal O}(1) \norm{v}_{H_{\Phi_1}(V)} e^{\frac{1}{h}(\Phi(x) - \frac{1}{C})}, \quad x\in V_1,
\endeq
where, similarly to (\ref{eq2.9}), we have
\begeq
\label{eq4.8}
v_z(x) = \frac{1}{(2 \pi h)^n} e^{\frac{i}{h}(x-z)\cdot \theta(x,z)} v(z)\chi(z) {\rm det}\,(\partial_{\overline{z}}\theta(x,z)) \in {\rm Hol}(V)
\endeq
is well localized at the point $z \in W_1$, see (\ref{eq2.6}).  Combining (\ref{eq4.6.2}), (\ref{eq4.7}), and (\ref{eq4.3}), we get
\begin{multline}
\label{eq4.9}
(\widetilde{\Pi}_{V} u,v)_{H_{\Phi}(V)} \\ = \int_{W_1}\!\int_{V_1} \widetilde{\Pi}_{V_2} u(x) \overline{v_z(x)} e^{-2\Phi(x)/h}\, L(dx)\, dz\, d\overline{z} + {\cal O}(1) e^{-\frac{1}{C h}}
\norm{u}_{H_{\Phi}(V)} \norm{v}_{H_{\Phi_1}(V)}.
\end{multline}
Let us rewrite (\ref{eq4.9}) as follows,
\begeq
\label{eq4.11}
(\widetilde{\Pi}_{V} u,v)_{H_{\Phi}(V)} = \int_{W_1} (\widetilde{\Pi}_{V_2} u,v_z)_{H_{\Phi}(V_1)} dz\, d\overline{z} + {\cal O}(1) e^{-\frac{1}{C h}} \norm{u}_{H_{\Phi}(V)} \norm{v}_{H_{\Phi_1}(V)}.
\endeq
When proving Theorem \ref{reproducing_scalar_product}, it will be convenient to work with the decomposition (\ref{eq4.11}), in view of the good localization properties of the holomorphic functions $v_z$, for $z\in W_1$.

\bigskip
\noindent
The crucial role in the proof is played by the following observation. 

\begin{prop}
\label{prop:two_good_contours}
Given $z\in V$, let us set for some $\delta >0$ small, 
\begeq
\label{eq4.11.1}
F_{z}(\widetilde{x}) = \Phi(\overline{\widetilde{x}}) - \delta \abs{\widetilde{x} - \overline{z}}^2,\quad \widetilde{x} \in \rho(V). 
\endeq
Let $G_z$ be the following real analytic plurisubharmonic function: 
\begeq
\label{eq4.12}
G_z(x,\widetilde{x},y,\widetilde{y}) = 2{\rm Re}\, \Psi(x,\widetilde{y}) - 2{\rm Re}\, \Psi(y,\widetilde{y}) + \Phi(y) + F_{z}(\widetilde{x}) - 2{\rm Re}\, \Psi(x,\widetilde{x}).
\endeq
Then $G_{z}$ has a non-degenerate critical point at $(z,\overline{z},z,\overline{z})$ of signature $(4n,4n)$, with the critical value equal to $0$. Furthermore, the following two submanifolds of $V\times \rho(V)\times V \times
\rho(V) \subset \comp^{4n}$ are good contours for $G_{z}$ in a neighbourhood of $(z,\overline{z},z,\overline{z})$, in the sense that they are both contours of maximal real dimension $4n$ passing through the critical point, along which the Hessian of $G_z$ is negative definite:
\begin{enumerate}
\item The contour 
\begeq
\label{eq4.12.1}
\Gamma_{V}\times \Gamma_{V}=\{(x,\widetilde{x},y,\widetilde{y});\,\widetilde{x} = \overline{x},\,\,\widetilde{y} = \overline{y},\,\, x\in V,\,\,y\in V\}
\endeq
\item The composed contour
\begeq
\label{eq4.12.2}  
\{(x,\widetilde{x},y,\widetilde{y});\, (y,\widetilde{x}) \in \Gamma_{V},\, (x,\widetilde{y}) \in \Gamma(y, \widetilde{x})\}. 
\endeq
Here $\Gamma(y,\widetilde{x}) \subset \comp^{2n}_{x,\widetilde{y}}$ is a good contour for the pluriharmonic function $(x,\widetilde{y}) \mapsto {\rm Re}\,\varphi(y,\widetilde{x};x,\widetilde{y})$ described in {\rm (\ref{eq3.11})}, {\rm (\ref{eq3.12})}.
\end{enumerate} 
\end{prop}
\begin{proof}
Let us observe first that the two contours clearly pass through the point $(z,\overline{z},z,\overline{z})$ and that $G_{z}(z,\overline{z},z,\overline{z})=0$, in view of (\ref{eq4.11.1}), (\ref{eq3.1}). In order to show that $(z,\overline{z},z,\overline{z})$ is a  non-degenerate critical point of signature $(4n,4n)$, it suffices, in view of the plurisubharmonicity of $G_z(x,\widetilde{x},y,\widetilde{y})$, to observe that, using (\ref{eq3.2}), (\ref{eq4.11.1}), we have
\begeq
\label{eq4.13}
G_{z}(x,\overline{x},y,\overline{y}) \leq -\frac{1}{C}\abs{y-x}^2 - \delta\abs{x-z}^2 \leq -\frac{1}{C}\abs{x-z}^2 - \frac{1}{C}\abs{y-z}^2.
\endeq
This establishes at the same time that the contour (\ref{eq4.12.1}) is a good contour for $G_z$. It only remains to prove that the second submanifold given in (\ref{eq4.12.2}) also defines a good contour. To this end, let us write, using (\ref{eq3.3}), (\ref{eq4.12}),
\begeq
\label{eq4.14}
G_{z}(x,\widetilde{x},y,\widetilde{y}) = 2{\rm Re}\, \varphi(y,\widetilde{x}; x,\widetilde{y}) - 2{\rm Re}\, \Psi(y,\widetilde{x}) + \Phi(y) + F_{z}(\widetilde{x}).
\endeq
Using (\ref{eq3.11}), (\ref{eq4.11.1}), (\ref{eq3.1}), we get therefore for $(y,\widetilde{x})\in \Gamma_{V}$, $(x,\widetilde{y}) \in \Gamma(y,\widetilde{x})$,
\begin{multline}
\label{eq4.15}
G_{z}(x,\widetilde{x},y,\widetilde{y}) \\ \leq -\frac{1}{C}\abs{y-x}^2 -\frac{1}{C}\abs{\widetilde{y}-\widetilde{x}}^2 - 2{\rm Re}\, \Psi(y,\widetilde{x}) + \Phi(y) + {\Phi}(\overline{\widetilde{x}}) - \delta\abs{\widetilde{x} - \overline{z}}^2 \\
= -\frac{1}{C}\abs{y-x}^2 - \frac{1}{C}\abs{\widetilde{y}-\widetilde{x}}^2 - \delta\abs{y - z}^2.
\end{multline}
It follows that
\begin{multline}
\label{eq4.16}
G_{z}(x,\widetilde{x},y,\widetilde{y}) \leq  -\frac{1}{C}\abs{x-z}^2 - \frac{1}{C}\abs{y-z}^2 - \frac{1}{C}\abs{\widetilde{x}-\overline{z}}^2 - \frac{1}{C}\abs{\widetilde{y}-\widetilde{x}}^2 \\
\leq -\frac{1}{C}\abs{x-z}^2 - \frac{1}{C}\abs{y-z}^2 - \frac{1}{C}\abs{\widetilde{x}-\overline{z}}^2 - \frac{1}{C}\abs{\widetilde{y}-\overline{z}}^2,
\end{multline}
which demonstrates that the composed contour (\ref{eq4.12.2}) is also good and concludes the proof. 
\end{proof}

\bigskip
\noindent
We are now ready to take a closer look at the scalar product $(\widetilde{\Pi}_{V_2}u,v_z)_{H_{\Phi}(V_1)}$, occuring in the right hand side of (\ref{eq4.11}). 
\begin{prop}\label{prop:repr-F_z}
There exists an open neighborhood $W_1\Subset V_1$ of $x_0$ such that, uniformly in $z\in W_1$, we have 
\begeq
\label{eq4.17}
(\widetilde{\Pi}_{V_2} u, v_z)_{H_{\Phi}(V_1)} = (u,v_z)_{H_{\Phi}(V_1)} + {\cal O}(1)e^{-\frac{1}{C h}} \norm{u}_{H_{\Phi}(V)} \abs{v(z)} e^{-\Phi(z)/h}.
\endeq
Here $v_z$ is given in {\rm {(\ref{eq4.8})}}. 
\end{prop} 
\begin{proof}
The scalar product in the Hilbert space of holomorphic functions $H_{\Phi}(V_1)$ can be expressed as follows,
\begeq
\label{eq4.18}
(f,g)_{H_{\Phi}(V_1)} = \int_{V_1} f(x) \overline{g(x)} e^{-\frac{2\Phi(x)}{h}}\, L(dx) = C_n  \int\!\!\!\int_{\Gamma_{V_1}} f(x) g^*(\widetilde{x}) e^{-\frac{2}{h}\Psi(x,\widetilde{x})}\, dx\, d\widetilde{x}.
\endeq
Here the contour $\Gamma_{V_1}$ is defined similarly to (\ref{eq4.2}) and $C_n$ is a numerical factor, depending on $n$ only, such that the Lebesgue measure $L(dx)$ on $\comp^n$ satisfies $L(dx) = C_n dx\, d\overline{x}$. In (\ref{eq4.18}) we have also set
\begeq
\label{eq4.19}
g^*(\widetilde{x}) = \overline{g(\overline{\widetilde{x}})} \in H_{\widehat{\Phi}}(\rho(V_1)), \quad \widehat{\Phi}(\widetilde{x}) = \Phi(\overline{\widetilde{x}}).
\endeq
Recalling (\ref{eq4.6.3}) and using (\ref{eq4.18}), we see that the scalar product $(\widetilde{\Pi}_{V_2} u, v_z)_{H_{\Phi}(V_1)}$ takes the form
\begeq
\label{eq4.20}
\frac{C_n}{h^n} \int\!\!\!\int_{\Gamma_{V_1}} \left(\int\!\!\!\int_{\Gamma_{V_2}} e^{\frac{2}{h}(\Psi(x,\widetilde{y}) - \Psi(y,\widetilde{y}))} a(x,\widetilde{y};h) u(y)\, dy\, d\widetilde{y}\right) v_z^*(\widetilde{x}) e^{-\frac{2}{h}\Psi(x,\widetilde{x})}\, dx\,d\widetilde{x}.
\endeq
Here using (\ref{eq4.8}), (\ref{eq2.6}), we observe that
\begeq
\label{eq4.21}
\abs{v_z^*(\widetilde{x})} \leq \frac{{\cal O}(1)}{h^n} \abs{v(z)} e^{-\Phi(z)/h} e^{F_z(\widetilde{x})/h}, \quad \widetilde{x} \in \rho(V_1),
\endeq
where $F_z$ is the strictly plurisubharmonic function in $\rho(V_1)$ given by
\begeq
\label{eq4.22}
F_z(\widetilde{x}) = \widehat{\Phi}(\widetilde{x}) - \delta\abs{\widetilde{x} - \overline{z}}^2,
\endeq
see also (\ref{eq4.11.1}). Combining (\ref{eq4.21}) with Proposition \ref{prop:ctrl_ptwise} we conclude that the absolute value of the holomorphic integrand in (\ref{eq4.20})
\begeq
\label{eq4.23}
V_1 \times \rho(V_1) \times V_2 \times \rho(V_2) \ni (x,\widetilde{x}, y,\widetilde{y}) \mapsto e^{\frac{2}{h}(\Psi(x,\widetilde{y}) - \Psi(y,\widetilde{y}))}  a(x,\widetilde{y};h) u(y) v_z^*(\widetilde{x}) e^{-\frac{2}{h}\Psi(x,\widetilde{x})}
\endeq
does not exceed 
\begeq
\label{eq4.24}
\frac{{\cal O}(1)}{h^{2n}} \norm{u}_{H_{\Phi}(V)} \abs{v(z)} e^{-\Phi(z)/h} e^{G_z(x,\widetilde{x}, y,\widetilde{y})/h}. 
\endeq
Here the plurisubharmonic function $G_z(x,\widetilde{x}, y,\widetilde{y})$ has been defined in (\ref{eq4.12}), and the contour of integration 
$\Gamma_{V_1} \times \Gamma_{V_2}$ in (\ref{eq4.20}) is therefore good for $G_z$, in view of Proposition \ref{prop:two_good_contours}. In particular, only a small neighborhood of the critical point $(z,\overline{z},z,\overline{z})$ gives a contribution that is not exponentially small to the integral (\ref{eq4.20}). In view of (\ref{eq4.11.1}), (\ref{eq4.12}), let us also remark that $G_z = G_{x_0} + {\cal O}(\delta\abs{x_0 -z})$.

\bigskip
\noindent
We shall now carry out a contour deformation in (\ref{eq4.20}), making use of Proposition \ref{prop:two_good_contours}. When doing so, let us recall from~\cite[Chapter 3]{Sj82},~\cite[Proposition 3.5]{Delort} that all good contours are homotopic, with the homotopy through good contours. As explained in~\cite[Chapter 3]{Sj82}, a homotopy between two good contours is obtained by working in the Morse coordinates in a neighborhood of the critical point. An application of the Stokes formula and Proposition \ref{prop:two_good_contours} allow us therefore to conclude that there exists a small open neighborhood $W_1 \Subset V_1$ of $x_0$ such that for all $z\in W_1$, the integral (\ref{eq4.20}) is equal to the integral
\begeq
\label{eq4.25}
C_n \int\!\!\!\int_{\Gamma_{V_1}} \left(\frac{1}{h^n} \int\!\!\!\int_{\Gamma(y,\widetilde{x})\cap (V_1\times \rho(V_1))} e^{\frac{2}{h} \varphi(y,\widetilde{x}; x,\widetilde{y})} a(x,\widetilde{y};h) dx\, d\widetilde{y}\right) u(y) v_{z}^*(\widetilde{x}) e^{-\frac{2}{h} \Psi(y,\widetilde{x})} dy\, d\widetilde{x},
\endeq
modulo an error term of the form
\begeq
\label{eq4.26}
{\cal O}(1) \norm{u}_{H_{\Phi}(V)} \abs{v(z)} e^{-\Phi(z)/h} e^{-\frac{1}{C h}}.
\endeq
Here we have also used (\ref{eq4.24}). An application of Proposition \ref{prop:constr-a} shows that the integral (\ref{eq4.25}) is equal to
\begeq
\label{eq4.27}
(u,v_z)_{H_{\Phi}(V_1)} + {\cal O}(1)e^{-\frac{1}{C h}} \norm{u}_{H_{\Phi}(V)} \abs{v(z)} e^{-\Phi(z)/h},
\endeq
which completes the proof. 
\end{proof}

\bigskip
\noindent
{\it Remark}. The advantage of exploiting the resolution of the identity given in Proposition \ref{Fourier_inv} is due precisely to the fact that it is thanks to it that we are able to reduce the study of the scalar product $(\widetilde{\Pi}_{V} u,v)_{H_{\Phi}(V)}$ to a superposition of integrals over good contours --- see (\ref{eq4.11}), (\ref{eq4.20}).

\bigskip
\noindent
It is now easy to finish the proof of Theorem \ref{reproducing_scalar_product}. To this end, we let $W\Subset W_1$, where $W_1$ is as in Proposition \ref{prop:repr-F_z}. Combining (\ref{eq4.11}) with (\ref{eq4.17}) we get 
\begeq
\label{eq4.28}
(\widetilde{\Pi}_{V} u, v)_{H_{\Phi}(V)} = \int_{W_1} (u,v_z)_{H_{\Phi}(V_1)}\, dz\,d\overline{z} + {\cal O}(1)e^{-\frac{1}{C h}} \norm{u}_{H_{\Phi}(V)} \norm{v}_{H_{\Phi_1}(V)}.
\endeq
On the other hand, using (\ref{eq4.7}), we can write
\begeq
\label{eq4.29}
(u, v)_{H_{\Phi}(V)} = \int_{W_1} (u,v_z)_{H_{\Phi}(V_1)}\, dz\,d\overline{z} + {\cal O}(1)e^{-\frac{1}{C h}} \norm{u}_{H_{\Phi}(V)} \norm{v}_{H_{\Phi_1}(V)}.
\endeq
The proof of Theorem \ref{reproducing_scalar_product} is complete.

\section{End of the proof of Theorem \ref{Theorem1}}
\label{sec:end-proof}
\setcounter{equation}{0}
To conclude the proof of Theorem \ref{Theorem1}, we shall first pass from the scalar products in Theorem \ref{reproducing_scalar_product} to weighted $L^2$ norm estimates . To this end, let $\Phi_1 \in C(\Omega;\real)$ be such that
\begeq
\label{eq5.1}
\Phi_1 \leq \Phi \,\, \wrtext{in}\,\, \Omega, \quad \Phi_1 < \Phi \,\, \wrtext{on}\,\, \Omega\setminus \overline{W}.
\endeq
Let us notice that while the weighted space $H_{\Phi_1}(V)$ is not preserved by the action of the operator $\widetilde{\Pi}_V$ in (\ref{eq4.1}), we still have
\begeq
\label{eq5.2}
\widetilde{\Pi}_V = {\cal O}(1): H_{\Phi_1}(V) \rightarrow H_{\Phi_2}(V),
\endeq
where similarly to (\ref{eq5.1}), the weight function $\Phi_2 \in C(\Omega;\real)$ satisfies
\begeq
\label{eq5.3}
\Phi_2 \leq \Phi \,\, \wrtext{in}\,\, \Omega, \quad \Phi_2 < \Phi \,\, \wrtext{on}\,\, \Omega\setminus \overline{W}.
\endeq
Indeed, let us write $\Phi_1 = \Phi - \psi_1$, $\psi_1 \geq 0$, with strict inequality on $\Omega\setminus\overline{W}$. Using (\ref{eq3.2}) together with the Schur test, we obtain (\ref{eq5.2}) with $\Phi_2 = \Phi - \psi_2$, where $0\leq \psi_2 \in C(\Omega;\real)$ is the infimal convolution
\begeq
\label{eq5.4}
\psi_2(x) = \inf_{y\in V} \left(\frac{\abs{x-y}^2}{2C} + \psi_1(y)\right).
\endeq
Here $C>0$ is sufficiently large. It is therefore clear that (\ref{eq5.3}) holds.

\medskip
\noindent
Let $u\in H_{\Phi_1}(V)$, where $\Phi_1 \in C(\Omega;\real)$ satisfies (\ref{eq5.1}), and let us apply Theorem \ref{reproducing_scalar_product}, with $v = (\widetilde{\Pi}_V -1)u \in H_{{\rm max}(\Phi_1,\Phi_2)}(V)$, and ${\rm max}(\Phi_1,\Phi_2)$ in place of $\Phi_1$. We obtain, using also (\ref{eq5.2}),
\begeq
\label{eq5.5}
\norm{(\widetilde{\Pi}_V -1)u}_{H_{\Phi}(V)} \leq {\cal O}(1) e^{-\frac{1}{Ch}} \norm{u}_{H_{\Phi_1}(V)}.
\endeq
The estimate (\ref{eq5.5}) is very close to the approximate reproducing property for $\widetilde{\Pi}_V$ that we seek but we still need to free ourselves from the auxiliary weight $\Phi_1$. This will be accomplished by $\overline{\partial}$--surgery. Without loss of generality, in what follows we shall assume therefore that the bounded open set $V$ is pseudoconvex, and we may even choose it to be a ball centered at $x_0$.

\bigskip
\noindent
Let $U \Subset W \Subset V$ be an open neighborhood of $x_0$ with $C^{\infty}$--boundary. Given $u\in H_{\Phi}(V)$, we shall estimate
\begeq
\label{eq5.6}
\norm{(\widetilde{\Pi}_V -1)u}_{H_{\Phi}(U)}.
\endeq
When doing so, let $\Phi_1 \in C^{\infty}(\Omega;\real)$ be such that
\begeq
\label{eq5.7}
\Phi_1 = \Phi \,\, \wrtext{in}\,\, \overline{W}, \quad \Phi_1 < \Phi \,\, \wrtext{on}\,\, \Omega\setminus \overline{W},
\endeq
with $\norm{\Phi - \Phi_1}_{C^2(\overline{V})}$ small enough. In particular, $\Phi_1$ is strictly plurisubharmonic in $V$, see also (\ref{eq2.1}), so that
\begeq
\label{eq5.8}
\sum_{j,k=1}^n \frac{\partial^2 \Phi_1}{\partial x_j \partial \overline{x}_k}(x) \xi_j \overline{\xi}_k \geq \frac{\abs{\xi}^2}{{\cal O}(1)}, \quad x\in V, \quad \xi \in \comp^n.
\endeq
Let $\chi \in C^{\infty}_0(W; [0,1])$ be such that $\chi = 1$ in a neighborhood of $\overline{U}$. We shall also need an auxiliary weight $\Phi_3\in C^{\infty}(V;\real)$ such that
\begeq
\label{eq5.9}
\Phi_3(x) \leq \Phi_1(x) \leq \Phi(x),\quad x\in V,
\endeq
which furthermore satisfies
\begeq
\label{eq5.10}
\Phi_3 = \Phi\,\, \wrtext{near}\,\,{\rm supp}\,(\nabla \chi),
\endeq
\begeq
\label{eq5.11}
\Phi_3 < \Phi\,\, \wrtext{near}\,\, \overline{U}.
\endeq
We may also arrange so that $\Phi_3$ is strictly plurisubharmonic in $V$,
\begeq
\label{eq5.12}
\sum_{j,k=1}^n \frac{\partial^2 \Phi_3}{\partial x_j \partial \overline{x}_k}(x) \xi_j \overline{\xi}_k \geq \frac{\abs{\xi}^2}{{\cal O}(1)}, \quad x\in V, \quad \xi \in \comp^n.
\endeq

\medskip
\noindent
When estimating (\ref{eq5.6}), we write
$$
u = \chi u + (1-\chi)u, \quad u\in H_{\Phi}(V).
$$
Here
$$
\overline{\partial}(\chi u) = u\overline{\partial}\chi
$$
satisfies
\begeq
\label{eq5.13}
\norm{\overline{\partial}(\chi u)}_{L^2_{\Phi_3}(V)} \leq {\cal O}(1) \norm{u}_{H_{\Phi}(V)},
\endeq
in view of (\ref{eq5.10}). By an application of H\"ormander's $L^2$-estimate for the $\overline{\partial}$--equation in the pseudoconvex open set $V$ for the weight $\Phi_3$ (\cite[Proposition 4.2.5]{Horm_Conv}), there exists $w\in L^2_{{\Phi_3}}(V)$ such that
\begeq
\label{eq5.14}
\overline{\partial} w = \overline{\partial}(\chi u),
\endeq
with
\begeq
\label{eq5.15}
\norm{w}_{L^2_{\Phi_3}(V)} \leq {\cal O}(h^{1/2}) \norm{\overline{\partial}(\chi u)}_{L^2_{\Phi_3}(V)} \leq {\cal O}(h^{1/2}) \norm{u}_{H_{\Phi}(V)}.
\endeq
Here we have also used (\ref{eq5.13}. Using (\ref{eq5.7}), (\ref{eq5.9}), and (\ref{eq5.15}), we see that the function $\chi u -w\in {\rm Hol}(V)$ satisfies
\begeq
\label{eq5.16}
\norm{\chi u -w}_{H_{\Phi_1}(V)} \leq \|\chi u\|_{L^2_{\Phi_1}(V)}+\|w\|_{L^2_{\Phi_1}(V)}={\cal O}(1) \norm{u}_{H_{\Phi}(V)},
\endeq
and therefore by (\ref{eq5.5}) we conclude that
\begeq
\label{eq5.17}
\norm{(\widetilde{\Pi}_V -1)(\chi u - w)}_{H_{\Phi}(V)}\leq {\cal O}(1) e^{-\frac{1}{Ch}} \norm{u}_{H_{\Phi}(V)}.
\endeq
Next, similarly to (\ref{eq4.6.1}), using (\ref{eq3.2}), we obtain that
\begeq
\label{eq5.18}
\norm{(\widetilde{\Pi}_V -1)(1-\chi)u}_{L^2_{\Phi}(U)} \leq {\cal O}(1) e^{-\frac{1}{Ch}}\norm{u}_{H_{\Phi}(V)}.
\endeq
We finally come to estimate the norm $\norm{(\widetilde{\Pi}_V - 1)w}_{L^2_{\Phi}(U)}$, and we remark first that in view of (\ref{eq5.11}), (\ref{eq5.15}), we have
\begeq
\label{eq5.19}
\norm{w}_{L^2_{\Phi}(U)} \leq {\cal O}(1) e^{-\frac{1}{Ch}}\norm{u}_{H_{\Phi}(V)}.
\endeq
Next, let $U \Subset U_1 \Subset W$ be such that we still have
\begeq
\label{eq5.20}
\Phi_3 < \Phi \,\, \wrtext{on}\,\, \overline{U_1},
\endeq
and let $\chi_{U_1}$ stand for the characteristic function of $U_1$. Using (\ref{eq3.2}), (\ref{eq5.9}), and (\ref{eq5.15}), we get
\begin{multline}
\label{eq5.21}
\norm{\widetilde{\Pi}_V w}_{L^2_{\Phi}(U)} \leq \norm{\widetilde{\Pi}_V (1-\chi_{U_1}) w}_{L^2_{\Phi}(U)}
+ \norm{\widetilde{\Pi}_V \chi_{U_1} w}_{L^2_{\Phi}(U)} \\
\leq {\cal O}(1) e^{-\frac{1}{Ch}} \norm{w}_{L^2_{\Phi}(V)} + \norm{\widetilde{\Pi}_V \chi_{U_1} w}_{L^2_{\Phi}(U)}
\leq {\cal O}(1) e^{-\frac{1}{Ch}} \norm{u}_{H_{\Phi}(V)} + \norm{\widetilde{\Pi}_V \chi_{U_1} w}_{L^2_{\Phi}(U)} \\
\leq {\cal O}(1) e^{-\frac{1}{Ch}} \norm{u}_{H_{\Phi}(V)} + {\cal O}(1)\norm{\chi_{U_1} w}_{L^2_{\Phi}(V)} \leq {\cal O}(1) e^{-\frac{1}{Ch}} \norm{u}_{H_{\Phi}(V)}.
\end{multline}
Here in the final estimate we have also used (\ref{eq5.20}) and (\ref{eq5.15}).

\bigskip
\noindent
Combining (\ref{eq5.17}), (\ref{eq5.18}), (\ref{eq5.19}), and (\ref{eq5.21}), we get
\begeq
\label{eq5.22}
\norm{(\widetilde{\Pi}_V -1)u}_{H_{\Phi}(U)} \leq {\cal O}(1) e^{-\frac{1}{Ch}} \norm{u}_{H_{\Phi}(V)}.
\endeq
The proof of Theorem \ref{Theorem1} is complete.

\end{document}